\documentclass[11pt]{article}
\usepackage[margin=0.75in]{geometry}

\usepackage{amsthm}
\usepackage{amssymb}
\usepackage{amsmath}
\usepackage{comment}
\usepackage{mathtools}
\usepackage{thm-restate}
\usepackage{url} 
\usepackage{hyperref}
\usepackage[noabbrev,capitalise]{cleveref}
\usepackage{verbatim}

\usepackage[utf8]{inputenc}

\usepackage{color}
\usepackage[normalem]{ulem}

\usepackage[shortlabels]{enumitem}

\usepackage[dvipsnames, table]{xcolor}

\numberwithin{equation}{section}

\newcommand{\eps}{\varepsilon}

\newcounter{i}
\theoremstyle{plain}
\newtheorem{thm}{Theorem}[section]
\newtheorem{lem}[thm]{Lemma}

\newtheorem{conj}[thm]{Conjecture}

{\noindent \emph{Proof.} {}{#1}{}}{\hfill
	$\Diamond$\vspace{1em}}

\theoremstyle{plain} 
\newcommand{\thistheoremname}{}
\newtheorem{genericthm}{\thistheoremname}

\theoremstyle{definition}
\newtheorem{definition}[thm]{Definition}

\newcommand{\cA}{\mathcal{A}}
\newcommand{\cQ}{\mathcal{Q}} 
\newcommand{\cD}{\mathcal{D}} 
\newcommand{\cB}{\mathcal{B}}
\newcommand{\cE}{\mathcal{E}}

\newcommand{\cM}{\mathcal{M}}
\newcommand{\cS}{\mathcal{S}}

\newcommand{\cH}{\mathcal{H}}
\newcommand{\cX}{\mathcal{X}}
\newcommand{\cP}{\mathcal{P}}
\newcommand{\cG}{\mathcal{G}}

\newcommand{\Prob}[1]{\ensuremath{%
\mathbb P\left[#1\right]
}}
\newcommand{\ProbCond}[2]{\ensuremath{%
    \mathbb P\left[#1\:\middle|\:#2\right]
  }}

\newcommand{\Expect}[1]{\ensuremath{%
\mathbb E\left[#1\right]
}}

\newcommand\filledcirc{\ensuremath{{\bullet}\mathllap{\circ}}}
\newcommand\hoclique{\ensuremath{\text{h}}}
\newcommand\vertclique{\ensuremath{{\text{v}}}}

\def\COMMENT#1{}

\title{Thresholds for $(n,q,2)$-Steiner Systems via Refined Absorption}

\author{
Michelle Delcourt
\thanks{Department of Mathematics, Toronto Metropolitan University (formerly named Ryerson University),
Toronto, Ontario M5B 2K3, Canada {\tt mdelcourt@torontomu.ca}. Research supported by NSERC under Discovery Grant No. 2019-04269.}
\and
Tom Kelly
\thanks{School of Mathematics, Georgia Institute of Technology, Atlanta, GA 30332, USA {\tt tom.kelly@gatech.edu}.
Research supported by the National Science Foundation under Grant No. DMS-2247078.}
\and
Luke Postle
\thanks{Combinatorics and Optimization Department,
University of Waterloo, Waterloo, Ontario N2L 3G1, Canada {\tt lpostle@uwaterloo.ca}. Partially supported by NSERC
under Discovery Grant No. 2019-04304.}}
\date{February 29, 2024}

\begin{document}

\maketitle

\begin{abstract} 



We prove that if $p \geq n^{-(q-6)/2}$, then asymptotically almost surely the binomial random $q$-uniform hypergraph $\mathcal{G}^{(q)}(n,p)$ contains an $(n,q,2)$-Steiner system, provided $n$ satisfies the necessary divisibility conditions.
\end{abstract}

\section{Introduction}

A \textit{Steiner system} with parameters $(n,q,r)$ or an \textit{$(n,q,r)$-Steiner system} is a set $S$ of $q$-subsets of an $n$-set $X$ such that every $r$-subset of $X$ belongs to exactly one element of $S$.  More generally, a \textit{design} with parameters $(n,q,r,\lambda)$ or an \textit{$(n,q,r,\lambda)$-design} is a set $S$ of $q$-subsets of an $n$-set $X$ such that every $r$-subset of $X$ belongs to exactly $\lambda$ elements of $S$.  This type of structure is one of the most fundamental objects of Combinatorial Design Theory.  For $r = 2$, a design with parameters $(n,q,2,\lambda)$ is also called a \textit{balanced incomplete block design (BIBD)} or a \textit{$2$-design}, and for $r = 2$ and $q = 3$, an $(n,3,2)$-Steiner system is a \textit{Steiner triple system} of order $n$.

In order for an $(n,q,r)$-Steiner system to exist, $n$ must satisfy some obvious \textit{divisibility conditions}, namely that $\binom{q - i}{r - i} \mid \binom{n - i}{r - i}$ for all $i \in \{0, \dots, r - 1\}$.  
One of the oldest theorems in combinatorics, due to Kirkman~\cite{K47} in 1847, is that the divisibility conditions are also sufficient for the existence of Steiner triple systems.  The ``Existence Conjecture'', dating back to the mid-1800s, asserts that for all $q$, $r$, and $\lambda$ and for sufficiently large $n$, the divisibility conditions are also sufficient for the existence of $(n,q,r, \lambda)$-designs.
In a celebrated series of papers in the 1970s, Wilson~\cite{WiI, WiII, WiIII} proved that $2$-designs exist for all sufficiently large $n$ satisfying the divisibility conditions.
In 1985, introducing the ``nibble method'', R\"odl~\cite{R85} proved the Erd\H{o}s--Hanani Conjecture~\cite{EH63} that ``approximate'' $(n,q,r)$-Steiner systems exist.
In 2014, Keevash~\cite{K14} proved the Existence Conjecture using ``randomized algebraic constructions''.  Shortly thereafter, in 2016 Glock, K\"uhn, Lo, and Osthus~\cite{GKLO16} gave a combinatorial proof of the Existence Conjecture via ``iterative absorption''.  All of these results have had a profound impact in combinatorics and beyond.
Recently, the first and third authors~\cite{DPI} introduced a technique called ``refined absorption'' and provided a new proof of the Existence Conjecture.  In this paper, we use this technique to show the existence of 2-designs in a probabilistic setting.
We need the following definition.

\begin{definition}
    Let $q \geq 1$ be an integer.  For every integer $n \geq q$ and real $p \in [0,1]$, let $\cG^{(q)}(n,p)$ be the random $n$-vertex $q$-uniform hypergraph in which each $q$-set is included as an edge independently with probability $p$.
    For an increasing property $\cP$ of $q$-uniform hypergraphs and an increasing sequence $(n_i : i \in \mathbb N)$, a function $p^*=p^*(n)$ is a \textit{threshold} if as $i \rightarrow \infty$,
\begin{equation*}
\Prob{ \cG^{(q)}(n_i,p) \in \mathcal{P} } \rightarrow 
 \begin{cases}
      0 & \text{ if }  p=o(p^*) \text{ and }\\
      1 & \text { if } p=\omega(p^*).
    \end{cases} 
 \end{equation*}   
  Although thresholds are not unique, every pair of thresholds is related by some multiplicative constant factor, so we will refer to \textit{the} threshold for a property as is common in the literature.
\end{definition}

For $r = 1$, it is easy to see that $(n,q,1)$-Steiner systems exist if and only if $q \mid n$.  However, the ``threshold'' version of this problem -- determining the threshold for $\cG^{(q)}(qn,p)$ to contain a $(qn,q,1)$-Steiner system, or equivalently, for $\cG^{(q)}(qn,p)$ to contain a \textit{perfect matching} -- was considered to be one of the most important problems in probabilistic combinatorics until its resolution by Johansson, Kahn, and Vu~\cite{JKV08} in 2008. This problem, known as ``Shamir's Problem'' (due to Erd\H{o}s \cite{E81b}), was also a major motivation behind the Kahn--Kalai Conjecture~\cite{KK07}, recently proved by Park and Pham~\cite{PP23}.  (The ``fractional'' version of this conjecture posed by Talagrand~\cite{Ta10} was proved slightly earlier by Frankston, Kahn, Narayanan, and Park~\cite{FKNP21}.)  The $q = 2$ case, concerning perfect matchings in random graphs, is a classic result of Erd\H{o}s and R\'enyi~\cite{ER66}.  Standard probabilistic arguments show that $n^{-q+1}\log n$ is the threshold for the property that every vertex of $\cG^{(q)}(qn,p)$ is contained in at least one edge, which in turn implies that the $(n,q,1)$-Steiner system threshold is at least $n^{-q + 1}\log n$.  Johansson, Kahn, and Vu~\cite{JKV08} proved a matching upper bound, and this fact also follows from the Park--Pham Theorem~\cite{PP23}.  In fact, much more is known.  Kahn~\cite{Ka19, Ka22} recently determined the \textit{sharp threshold} and even proved a \textit{hitting time} result for this problem. 

For $q > r > 1$, determining the threshold for $\cG^{(q)}(n, p)$ to contain an $(n,q,r)$-Steiner system (for $n$ satisfying the divisibility conditions) is more challenging, as any nontrivial upper bound on the threshold necessarily implies not only that these designs exist, but that they exist ``robustly''.  (Studying ``robustness'' of graph properties has been an important trend in extremal and probabilistic combinatorics -- see for example the survey of Sudakov~\cite{Su17}.)  Hence, any result of this type seemed out of reach until Keevash's~\cite{K14} breakthrough in 2014.  Nevertheless, in 2006, Johansson~\cite{Jo06} posed a related conjecture on the threshold problem for \textit{Latin squares}.  
In 2017, Simkin~\cite{S17} conjectured that $n^{-1}\log n$ is the threshold for $\cG^{(q)}(n,p)$ to contain an $(n,q,q-1)$-Steiner system for every integer $q > 1$, and Keevash posed the same problem for $q = 3$ (the case of Steiner triple systems) in his 2018 ICM talk. 
Similar to the $r = 1$ case, standard arguments show that $n^{-1}\log n$ is a lower bound on this threshold.
Utilizing the work of Keevash~\cite{K14} on the Existence Conjecture, Simkin~\cite{S17} showed that the threshold for $(n,q,q-1)$-Steiner systems for any fixed integer $q > 1$ is at most $n^{-\varepsilon}$ for some $\varepsilon > 0$ that depends on $q$.
Following the breakthroughs of Frankston, Kahn, Narayanan, and Park~\cite{FKNP21} and Park and Pham~\cite{PP23} on the Kahn--Kalai Conjecture~\cite{KK07}, a series of results in 2022 culminated in the resolution of the threshold problem for Steiner triple systems.
First, Sah, Sawhney, and Simkin~\cite{SSS23} proved an upper bound on the threshold of $n^{-1+o(1)}$; Kang, Kelly, K\"uhn, Methuku, and Osthus~\cite{KKKMO22} proved the better bound of $n^{-1}\log^2 n$. Subsequently Jain and Pham~\cite{JP22} and independently Keevash~\cite{K22} settled the problem, proving that $n^{-1}\log n$ upper bounds the threshold.
In fact, Jain and Pham~\cite{JP22} and Keevash~\cite{K22} first proved that $n^{-1}\log n$ upper bounds the Latin square threshold and used a reduction of Kang, Kelly, K\"uhn, Methuku, and Osthus~\cite{KKKMO22} to then derive the Steiner triple system threshold.

Besides the cases of $r = 1$ and of $r = 2$ and $q = 3$, the threshold problem for $(n, q, r)$-Steiner systems is wide open.  
Kang, Kelly, K\"uhn, Methuku, and Osthus~\cite{KKKMO22} conjectured that the threshold is $n^{-q+r}\log n$.  
For $p = o(n^{-q+r}\log n)$, asymptotically almost surely $\cG^{(q)}(n,p)$ will contain an $r$-set of vertices that is not contained in any edge and thus does not contain an $(n,q,r)$-Steiner system.  Hence, to prove this conjecture it suffices to show that $n^{-q+r}\log n$ upper bounds the threshold for $\cG^{(q)}(n,p)$ to contain an $(n,q,r)$-Steiner system, as follows.

\begin{conj}[Kang, Kelly, K\"uhn, Methuku, and Osthus \cite{KKKMO22}]\label{conj:KKKMO} For every $q,r\in  \mathbb{N}$ with $q>r$, the following holds. 
If $\binom{q - i}{r - i} \mid \binom{n - i}{r - i}$ for all $i \in \{0, \dots, r-1\}$ and $p = \omega(n^{-q+r}\log n)$, then
asymptotically almost surely $\cG^{(q)}(n, p)$ contains an $(n,q,r)$-Steiner system.
\end{conj}

The main result of this paper is progress on the $r=2$ case of Conjecture~\ref{conj:KKKMO} as follows.

\begin{thm}\label{thm:ExistenceSpread} For every integer $q > 2$, the following holds:  
If $q - 1 \mid n - 1$ and $\binom{q}{2}\mid \binom{n}{2}$ and $p\geq n^{-(q-6)/2}$, then asymptotically almost surely $\cG^{(q)}(n,p)$ contains an $(n, q, 2)$-Steiner system.
\end{thm}

Besides the $q = 3$ case, no nontrivial upper bound on the threshold for $(n,q,2)$-Steiner systems was known prior to our work.  
The exponent of $n$ in the lower bound for $p$ in Theorem~\ref{thm:ExistenceSpread} is essentially a factor of two away from the conjectured value.
For $p = 1$, Theorem~\ref{thm:ExistenceSpread} is equivalent to Wilson's~\cite{WiI, WiII, WiIII} result that the Existence Conjecture holds for $r = 2$.

For the rest of the paper, we will view Steiner systems through the lens of graph decompositions.  
For graphs $F$ and $G$, an \textit{$F$-packing} of $G$ is a collection $\cH$ of pairwise edge-disjoint subgraphs of $G$, each isomorphic to $F$, and an \textit{$F$-decomposition} of $G$ is an $F$-packing $\cH$ of $G$ in which every edge of $G$ is in an element of $\cH$.  
An $(n,q,2)$-Steiner system corresponds to a $K_q$-decomposition of $K_n$, where $K_n$ denotes the complete graph on $n$ vertices.  More generally, an $(n,q,r)$-Steiner system corresponds to a decomposition of the edges of the complete $n$-vertex $r$-uniform hypergraph $K^r_n$ into subhypergraphs isomorphic to $K^r_q$.  
We say a graph $G$ is \textit{$K_q$-divisible} if the degree $d(v)$ of every vertex $v\in V(G)$ is divisible by $q - 1$ and $e(G)$ is divisible by $\binom{q}{2}$.  Note that $K_n$ is $K_q$-divisible if and only if $q - 1 \mid n - 1$ and $\binom{q}{2} \mid \binom{n}{2}$.

To prove Theorem~\ref{thm:ExistenceSpread}, we employ the Park--Pham Theorem~\cite{PP23}.  In fact, the fractional version of Frankston, Kahn, Narayanan, and Park~\cite{FKNP21} suffices.  To use this theorem, our main objective is to show the existence of a sufficiently ``spread'' probability distribution on the $K_q$-decompositions of $K_n$ (provided $K_n$ is $K_q$-divisible). In this context, the definition of \textit{spread} is as follows.

\begin{definition}
    A probability distribution on $K_q$-decompositions $\cH$ of a graph $G$ is \textit{$\sigma$-spread} if
    \begin{equation*}
        \Prob{\cS \subseteq \cH} \leq \sigma^{|\cS|} \text{ for all $K_q$-packings $\cS$ of $G$.} 
    \end{equation*}
\end{definition}

The Park--Pham Theorem implies that for some absolute constant $K > 0$, if there exists a $\sigma_n$-spread probability distribution on $K_q$-decompositions of $K_n$ and $p \geq K\sigma_n\log n$, then $\cG^{(q)}(n, p)$ asymptotically almost surely contains an $(n, q, 2)$-Steiner system.  Hence, Theorem~\ref{thm:ExistenceSpread} follows immediately from the following result.

\begin{thm}\label{thm:spread-decomposition}
  For every integer $q > 2$, there exists $\beta > 0$ such that the following holds for all sufficiently large $n$.  If $K_n$ is $K_q$-divisible, then there exists an $\left(n^{-(q - 6)/2 - \beta}\right)$-spread probability distribution on $K_q$-decompositions of $K_n$.
\end{thm}

In fact, Theorem~\ref{thm:spread-decomposition} implies a slightly stronger form of Theorem~\ref{thm:ExistenceSpread} with the condition ``$p \geq n^{-(q-6)/2}$'' replaced with ``$p \geq n^{-(q - 6)/2 - \beta}$'' for some $\beta > 0$ sufficiently small with respect to $q$; however, we chose not to state this in Theorem~\ref{thm:ExistenceSpread} for brevity, since this bound is likely not optimal.  In fact, it seems likely that the ``$6$'' can be improved to ``$2$'' or even to ``$1 + o(1)$'', but we did not pursue optimizing this constant so as to present a shorter proof.  The multiplicative $1/2$ factor on the other hand seems to be a theoretical limit for our methods, which we discuss further in the next section.


We conclude this section with some open questions.  First, we note that our methods easily imply that for all integers $q > r\ge 1$, there exists some $\varepsilon > 0$ depending on $q$ and $r$ such that Conjecture~\ref{conj:KKKMO} holds with $p=n^{-\varepsilon}$ (that is, this gives an independent proof of Simkin's result and generalizes it to all $q$). It would be interesting to prove explicit bounds with an eye towards proving Conjecture~\ref{conj:KKKMO}.

Since the proof of the Park--Pham Theorem is nonconstructive, we do not have an efficient algorithm for finding the $(n, q, 2)$-Steiner system promised by Theorem~\ref{thm:ExistenceSpread}. As we discuss in the next section, for $p \geq n^{-(q - 6)/4}$, we believe our methods would yield a randomized polynomial-time algorithm for finding an $(n,q,2)$-Steiner system asymptotically almost surely in $\cG^{(q)}(n,p)$. It would be interesting to determine whether such an algorithm exists in the setting of Conjecture~\ref{conj:KKKMO}.  Even the $r = 1$ case of this problem is open.

Recently, the first and third authors~\cite{DPII} proved the existence of high-girth Steiner systems.  It is plausible that Conjecture~\ref{conj:KKKMO} and Theorem~\ref{thm:ExistenceSpread} hold for high girth Steiner systems as well.  Similarly, the authors~\cite{DKPIII} recently proved new results for $K_q$-packings of random graphs, and it would be interesting to see whether the results of these two papers can be combined.  For example, given $p_1 = p_1(n)$ and $p_2 = p_2(n)$, when is it possible to find a $K_q$-packing of the binomial random graph $G\sim G(n,p_1)$ with small leave using cliques of $G$ chosen independently at random with probability $p_2$?  

The Nash-Williams Conjecture~\cite{N-W70} asserts that every $K_3$-divisible $n$-vertex graph with minimum degree at least $3n/4$ admits a $K_3$-decomposition, and it has been conjectured \cite{Gu91, BKLO16} that more generally a $K_q$-divisible $n$-vertex graph with minimum degree at least $(1 - 1/(q + 1))n$ admits a $K_q$-decomposition.  It would again be interesting to see whether this problem can be combined with those in this series.  For example, do such graphs admit high-girth decompositions, and can we find these decompositions using only a binomial random selection of the cliques?  We conjecture that in fact, for all $g$, if $G$ is an $n$-vertex $K_q$-divisible graph with minimum degree at least $(1 - 1/(q + 1))n$ and $p = \omega(n^{2-q}\log n)$, then asymptotically almost surely there exists a $K_q$-decomposition of $G$ of girth at least $g$ using only a binomial $p$-random selection of the copies of $K_q$ of $G$.

In the next section, we provide an overview of the proof of Theorem~\ref{thm:spread-decomposition} before outlining the rest of the paper.

\section{Proof Overview}\label{s:Overview}

Theorem~\ref{thm:ExistenceSpread} asserts the existence of $(n,q,2)$-Steiner systems in $\mathcal{G}^{(q)}(n,p)$ for some small $p$, and similarly, Theorem~\ref{thm:spread-decomposition} asserts the existence of a highly spread distribution on $(n,q,2)$-Steiner systems. It is natural to inquire how the existential proofs for $(n,q,2)$-Steiner systems fare in such settings. Wilson's proof~\cite{WiI,WiII,WiIII} uses a recursive algebraic construction and hence is not amenable due to its lack of randomness. Keevash's proof~\cite{K14} utilizes ``random algebraic constructions''; indeed, in theory his work yields a spread of $n^{-\varepsilon}$ for some $\varepsilon > 0$ (a la Simkin~\cite{S17} for the $q=r+1$ case), but the $\varepsilon$ is likely to be far from optimal (namely $\varepsilon$ is at most $2^{-\Omega(q)}$ say). The iterative absorption proof of Glock, K\"uhn, Lo, and Osthus~\cite{GKLO16} builds a $K_q$-decomposition by slowly restricting the `leftover' (i.e. undecomposed) edges to smaller and smaller sets of vertices. This inherently is not highly spread since the cliques used to decompose the smallest set of vertices will not be highly spread. 
Meanwhile, the techniques that resolved the $q=3$ case of Conjecture~\ref{conj:KKKMO} do not seem to readily extend to the case of general $q$ since they rely on the properties of Latin squares and their association with edge-colorings (and hence are specific to the $q=3$ case). 

Thus, the key to our proof of Theorem~\ref{thm:spread-decomposition} is to use the method of ``refined absorption'' developed by the first and third authors~\cite{DPI} in their new proof of the Existence Conjecture. While that proof does not immediately carry over to our settings, we are fortunately able to use the key technical theorem from~\cite{DPI} (see Theorem~\ref{thm:omni-absorber} below) as a black box. The novel (and certainly still non-trivial) work of this paper then is how to utilize the template provided by refined absorption and apply it to produce highly spread distributions of $(n,q,2)$-Steiner systems. 

\subsection{Proof of Existence via Refined Absorption}

To that end, we first outline the proof of Existence from~\cite{DPI} as well as the key technical definition and black box theorem before discussing the modifications that allow us to extract the true power of this new proof.
The iterative absorption proof of Existence uses \textit{absorbers}.  In this context, if $L$ is a $K_q$-divisible graph, we say a graph $A$ is a \emph{$K_q$-absorber} for $L$ if $V(L)\subseteq V(A)$ is independent in $A$ and both $A$ and $L\cup A$ admit $K_q$-decompositions.
The first key concept from~\cite{DPI} is that of an `omni-absorber', an object that absorbs \emph{all} possible leftovers.

\begin{definition}[Omni-Absorber]
Let $q \geq 3$ be an integer. Let $X$ be a graph. We say a graph $A$ is a \textit{$K_q$-omni-absorber} for $X$ with \emph{decomposition family} $\mathcal{H}$ and \emph{decomposition function} $\mathcal{Q}_A$ if $V(X)=V(A)$, $X$ and $A$ are edge-disjoint, $\mathcal{H}$ is a family of subgraphs of $X\cup A$ each isomorphic to $K_q$ such that $|E(H)\cap E(X)|\le 1$ for all $H\in\mathcal{H}$, and for every $K_q$-divisible subgraph $L$ of $X$, there exists $\mathcal{Q}_A(L)\subseteq \mathcal{H}$ that are pairwise edge-disjoint and such that $\bigcup \mathcal{Q}_A(L)=L\cup A$. 
\end{definition}

The next key concept is that of \emph{refinement} of an omni-absorber, that is that every edge of $X\cup A$ is in only constantly many cliques of the decomposition family as follows.

\begin{definition}[Refined Omni-Absorber]
Let $C\ge 1$ be real. We say a $K_q$-omni-absorber $A$ for a graph $X$ with decomposition family $\mathcal{H}$ is \emph{$C$-refined} if $|\{H\in \mathcal{H} : e\in E(H) \}| \le C$ for every edge $e\in X\cup A$.
\end{definition}

We may now state the key technical theorem from~\cite{DPI} which asserts that linearly efficient refined omni-absorbers exist. 
Here we recall just the graph version of the main omni-absorber theorem from~\cite{DPI} (which was more generally proved for hypergraphs).

\begin{thm}\label{thm:omni-absorber}
    For every integer $q \ge 3$, there exists a real $C\ge 1$ such that the following holds: If $X$ is a spanning subgraph of $K_n$ with $\Delta(X) \le \frac{n}{C}$ and we let $\Delta:= \max\left\{\Delta(X),\sqrt{n}\cdot \log n\right\}$, then there exists a $C$-refined $K_q$-omni-absorber $A \subseteq K_n$ for $X$ such that $\Delta(A)\le C \cdot \Delta$.
\end{thm}

\noindent
These linearly efficient omni-absorbers permitted a more streamlined proof of the Existence Conjecture whose outline we now recall from~\cite{DPI} but restricted to the graph case for simplicity:

\begin{enumerate}
    \item[(1)] `Reserve' a random subset $X$ of $E(K_n)$.
    \item[(2)] Construct a $K_q$-omni-absorber $A$ of $X$.
    \item[(3)] ``Regularity boost'' $K_n\setminus (A\cup X)$.
    \item[(4)] Apply ``nibble with reserves'' theorem to find a $K_q$-packing of $K_n\setminus A$ covering $K_n\setminus (A\cup X)$ and then extend this to a $K_q$-decomposition of $K_n$ by definition of omni-absorber.  
\end{enumerate}

We now discuss the difficulties in adapting this proof to the threshold or spread settings. The main difficulty for the proofs of Theorem~\ref{thm:ExistenceSpread} and Theorem~\ref{thm:spread-decomposition} lies in Step (2), namely we need to embed an omni-absorber inside $\mathcal{G}^{(q)}(n,p)$ for Theorem~\ref{thm:ExistenceSpread} and in a spread way for Theorem~\ref{thm:spread-decomposition}. Similarly, we will also require a threshold/spread version of nibble with reserves for Step (4). However, since we are not striving for the optimal value of $p$, this may be accomplished via a simple random sparsification argument which will only require a $p$ that is ${\rm polylog}~n$ times the conjectured value of $p$ (and hence is not the bottleneck for our proof).

\subsection{Building a Spread Omni-Absorber}

For Step (2), we have to ensure that the $K_q$'s used in the decomposition family of our omni-absorber can be embedded inside $\mathcal{G}^{(q)}(n,p)$ or in a spread way. With the ideas from this paper, it may be possible to amend every step of the proof of Theorem~\ref{thm:omni-absorber} to be done inside $\mathcal{G}^{(q)}(n,p)$ or in a spread way; however, it is far easier to use Theorem~\ref{thm:omni-absorber} as a black box and then seek to increase the spread using structures we call \textit{spread boosters}.

To that end, we view the omni-absorber $A$ from Theorem~\ref{thm:omni-absorber} as a template, and  we randomly embed private spread boosters for each clique in the decomposition family $\mathcal{H}$ of $A$. 
Here, a \textit{$K_q$-booster} is a graph with two disjoint $K_q$-decompositions $\mathcal{B}^{\circ}$ and $\mathcal{B}^{\filledcirc}$ (see Definition~\ref{def:Booster}).  
If a clique $Q\in \mathcal{H}$ is also in $\mathcal{B}^{\circ}\cup \mathcal{B}^{\filledcirc}$, say $\mathcal{B}^{\circ}$ without loss of generality, then we can replace $Q$ in the decomposition family $\mathcal{H}$ with $(\mathcal{B}^{\circ}\setminus \{Q\}) \cup \mathcal{B}^{\filledcirc}$ (the other cliques of the booster). Then if $Q$ would be used to decompose $L\cup A$, we instead use $\mathcal{B}^{\filledcirc}$, while if $Q$ is not used to decompose $L\cup A$, we use $\mathcal{B}^{\circ}\setminus \{Q\}$. In this way, the booster may be viewed as a `non-trivial' absorber for the clique $Q$.  
Since we are free to choose these boosters randomly, the new set of cliques will be more random than the original pre-determined clique (and hence we ``boosted the spread'').

For the proof then, we argue that the boosters are spread if chosen independently. However, we must use the Lov\'asz Local Lemma to ensure that the boosters are embedded disjointly. We argue via the Lov\'asz Local Lemma distribution that the spreadness of a fixed set of cliques is not too much larger for the disjoint case than if the choices were truly independent. 



The bottleneck for this argument turns out to be the rooted density of the boosters (see Definition~\ref{def:RootedDensity}). In this paper, we construct boosters of density $2/(q - 2)$ (see Lemma~\ref{lem:sparse-boosters-exist}) and this yields a value of $1/n^{(q-6)/2}$ for spreadness. With more effort, we could probably improve the proof to yield the reciprocal of the density in the exponent. Similarly, the booster density could perhaps be improved to $2/(q-1+\beta)$ for any small $\beta > 0$. 
 However, we have not pursued this in this paper for the sake of brevity.  As shown in \cite[Lemma 10.1]{DKPIII}, every non-isolated vertex of a $K_q$-booster is in  at least two cliques of each decomposition, and thus, a double counting argument shows that $K_q$-boosters will always have rooted density at least $2/q$.\COMMENT{
    Let $B$ be a $K_q$-booster with decompositions $\cH$ and $\cH'$.  We may assume the rooted booster is obtained from $B$ by rooting on some clique in $\cH$. We may also assume without loss of generality that $B$ has no isolated vertices. Hence every vertex is in at least two cliques of each decomposition.
Thus by hypergraph handshaking, 
    \begin{equation*}
        2|V(B)| \geq \sum_{v \in V(B) }\text{\# cliques in $\cH$ containing v} = q|\cH|.
    \end{equation*}
    Hence the rooted density is at least
    \begin{equation*}
        \frac{|\cH| - 1}{|V(B)| - q} \geq \frac{2|V(B)|/q - 1}{|V(B)| - 1} = \frac{2/q - 1/|V(B)|}{1 - 1/|V(B)|}.
    \end{equation*}
    The right side of the inequality above is decreasing in $|V(B)|$ with a limit of $2 / q$, as desired.
 } Hence, new techniques would be needed to prove Conjecture~\ref{conj:KKKMO} for general $q$. 

Finally, we remark that the booster embedding could have been done inside $G^{(q)}(n,p)$ directly using the techniques from~\cite{DKPIII}; however, that would require the existence of both decompositions of the boosters inside, which yields an extra factor of $2$ in the threshold exponent ($p \geq n^{-{(q - 6)}/{4}}$ vs $n^{-{(q - 6)}/{2}}$-spread) but with the benefit that it yields an algorithm for finding the $(n,q,2)$-Steiner system. By using the Park-Pham Theorem and spreadness, this mitigates the extra factor of $2$ at the cost of being non-constructive.


\subsection{Outline of Paper}\label{ss:Outline}


In Section~\ref{s:Previous}, we recall various necessary previous results for our main proofs. In Section~\ref{s:SpreadOmniAbsorber}, we prove a spread distribution on omni-absorbers modulo the construction of spread boosters which we do in Section~\ref{sec:spread-boosters}. In Section~\ref{s:SpreadNibble}, we prove a spread version of nibble with reserves. In Section~\ref{s:MainProof}, we then prove Theorem~\ref{thm:spread-decomposition}, from which Theorem~\ref{thm:ExistenceSpread} follows via the Park-Pham Theorem. 

\section{Preliminaries}\label{s:Previous}

In this section, we collect the other previous results we need.

\subsection{Reservoir}

For Step (1), we require the following lemma from~\cite{DPI}.

\begin{lem}\label{lem:reservoir}
  For every integer $q\geq 3$, there exists $\eps > 0$ such that the following holds for sufficiently large $n \in \mathbb N$.  If $G$ is an $n$-vertex graph with $\delta(G)\ge (1-\varepsilon)n$ and $1 / n^\eps \leq p \leq 1$, then there exists $X\subseteq G$ such that $\Delta(X)\le 2pn$ and for all $e\in E(G)\setminus E(X)$, there exist at least $\eps p^{\binom{q}{2}-1}n^{q-2}$ $K_q$'s in $X\cup \{e\}$ containing $e$.
\end{lem}

\subsection{The Design Hypergraph}

Rephrasing a decomposition problem in terms of a perfect matching of an auxiliary hypergraph is useful for Step (4). To that end, we have the following definition.

\begin{definition}[Design Hypergraph]
Let $F$ be a hypergraph. If $G$ is a hypergraph, then the \emph{$F$-design hypergraph of $G$}, denoted ${\rm Design}(G,F)$ 
is the hypergraph $\Gamma$ with $V(\Gamma)=E(G)$ and $E(\Gamma) = \{S\subseteq E(G): S \text{ is isomorphic to } F\}$.
\end{definition}

Note that $\mathrm{Design}(K_n, K_q)$ is $\binom{n - 2}{q - 2}$-regular with codegrees at most $\binom{n - 3}{q - 3}$.
For Step (3), we recall the following corollary of the Boosting Lemma of Glock, K\"uhn, Lo and Osthus~\cite{GKLO16}.

\begin{lem}\label{lem:RegBoost}
For every integer $q \geq 3$, there exists $\eps > 0$ such that the following holds for sufficiently large $n$:  If $J$ is an $n$-vertex graph with $\delta(J) \ge (1-\varepsilon)n$, then there exists a subhypergraph $\cD$ of ${\rm Design}(J,K_q)$ such that $d_\cD(v) = (1/2 \pm n^{-(q-2)/3})\left.\binom{n - 2}{q - 2}\middle.\right.$ for all $v\in V(\cD)$.
\end{lem}

\subsection{Nibble with Reserves}

For Step (4), we require the ``nibble with reserves'' theorem, but first a definition.

\begin{definition}[Bipartite Hypergraph]
We say a hypergraph $G=(A,B)$ is \emph{bipartite with parts $A$ and $B$} if $V(G)=A\cup B$ and  every edge of $G$ contains exactly one vertex from $A$. We say a matching of $G$ is \emph{$A$-perfect} if every vertex of $A$ is in an edge of the matching.
\end{definition}

Here is ``nibble with reserves'' from~\cite{DP22}. 

\begin{thm}\label{thm:NibbleReserves}
For every integer $r \ge 2$ and real $\beta \in (0,1)$, there exist an integer $D_{\beta}\ge 0$ and real $\alpha > 0$ such that following holds for all $D\ge D_{\beta}$: 
\vskip.05in
Let $G$ be an $r$-uniform (multi)-hypergraph with codegrees at most $D^{1-\beta}$ such that $G$ is the edge-disjoint union of $G_1$ and $G_2$ where $G_2=(A,B)$ is a bipartite hypergraph such that every vertex of $B$ has degree at most $D$ in $G_2$ and every vertex of $A$ has degree at least $D^{1-\alpha}$ in $G_2$, and $G_1$ is a hypergraph with $V(G_1)\cap V(G_2) = A$ such that every vertex of $G_1$ has degree at most $D$ in $G_1$ and every vertex of $A$ has degree at least $D\left(1- D^{-\beta}\right)$ in $G_1$.

Then there exists an $A$-perfect matching of $G$. 
\end{thm}


\subsection{The Lov\'asz Local Lemma Distribution}

To prove that spread boosters may be embedded disjointly in a manner almost as spread as if they were embedded independently at random, we require the following variant of the Lov\'asz Local Lemma~\cite{EL73} for the conditional distribution (see Haeupler, Saha, and Srinivasan~\cite{HSS11} and Jain and Pham~\cite{JP22}).

\begin{lem}\label{lem:conditionalLLL}
    Let $\{X_i\}_{i\in I}$ be independent random variables, and let $\{\cE_j\}_{j\in J}$ be events where each $\cE_j$ depends on a subset $S_j \subseteq I$ of the variables.  Let $\Gamma$ be a graph with vertex set $J$ where $S_j \cap S_{j'} = \emptyset$ for all distinct non-adjacent $j,j'\in J$.  Let $\mathbb P$ be the product measure on the random variables $\{X_i\}_{i\in I}$, and let $\cE$ be an event depending on a subset $S \subseteq I$.
    If $\Prob{\cE_j} \leq p$ for all $j \in J$ and $4p\Delta(\Gamma) \leq 1$, then 
    \begin{equation*}
        \ProbCond{\cE}{\bigcap_{j \in J}\overline{\cE_j}} \leq \Prob{\cE}\exp(6pN),
    \end{equation*}
    where $N$ is the number of events $\cE_j$ with $S_j \cap S \neq \emptyset$.
\end{lem}

\section{Spread Omni-Absorber Theorem}\label{s:SpreadOmniAbsorber}

The main ingredient in our proof of Theorem~\ref{thm:spread-decomposition} is a ``spread'' version of Theorem~\ref{thm:omni-absorber} as follows.

\begin{thm}[Refined Omni-Absorber Theorem for Graphs - Spread Version]\label{thm:OmniSpread}
    For every integer $q \ge 3$, the following holds for sufficiently large $n,C\ge 1$ such that $C \leq \sqrt{n} / {\log n}$.
    If $X$ is a spanning subgraph of $G \cong K_n$ with $\Delta(X) \le n / C$, then there exists a probability distribution over $K_q$-omni-absorbers $A \subseteq G$ for $X$ with decomposition family $\cH_A$ and decomposition function $\cQ_A$ such that $\Delta(A)\le C \cdot \max \left\{\Delta(X), \sqrt{n}\log n\right\}$ and 
    \begin{equation*}
        \Prob{\bigcup_{L\subseteq X}\{\cS \subseteq \cQ_A(L)\}} \le \left(\frac{1}{C^{1/4}n^{(q - 6)/2}}\right)^{|\cS|} \text{ for all $K_q$-packings $\cS$ of $G$.}
    \end{equation*}
\end{thm}
Here, the event $\bigcup_{L\subseteq X}\{\cS \subseteq \cQ_A(L)\}$ is the event that there exists some $K_q$-divisible $L \subseteq X$ for which $\cS$ is contained in the $K_q$-decomposition $\cQ_A(L)$ of $A \cup L$, where $A$ is a random $K_q$-omni-absorber.
Throughout this section, we use $A'$ to refer to an omni-absorber for $X$ with decomposition family $\cH_{A'}$ and decomposition function $\cQ_{A'}$ obtained by applying Theorem~\ref{thm:omni-absorber}.  To prove Theorem~\ref{thm:OmniSpread}, we use the Lov\'asz Local Lemma distribution (see Lemma~\ref{lem:conditionalLLL}) to find a distribution on a random omni-absorber $A$ obtained by replacing each copy of $K_q$ in $\cH_{A'}$ with a spread booster and updating the decomposition function accordingly.

\subsection{Spread Boosters}

When we randomly replace a copy of $K_q$ in the decomposition family $\cH_{A'}$ with a spread booster, it will be helpful to work with the following notion of a rooted booster.

\begin{definition}\label{def:Booster}Let $q > 1$ be an integer.
   \begin{itemize}
       \item A \textit{$K_q$-booster} is a graph $B$ along with two $K_q$-decompositions $\cB^{\circ}$ and $\cB^{\filledcirc}$ of $B$ such that $\cB^{\circ} \cap \cB^{\filledcirc} = \emptyset$.
       \item A tuple $(B, \cB^{\circ}, \cB^{\filledcirc}, R)$ is a \textit{rooted $K_q$-booster} if
       \begin{itemize}
           \item $B$ is a graph,
           \item $\cB^{\filledcirc}$ is a $K_q$-decomposition of $B$, called the \textit{off-decomposition},
           \item $R \cong K_q$, $V(R) \subseteq V(B)$, and $R$ is edge-disjoint from $B$, and
           \item $\cB^{\circ}$ is a $K_q$-decomposition of $B \cup R$ with $R \notin \cB^{\circ}$, called the \textit{on-decomposition}.
       \end{itemize}
   \end{itemize}    
\end{definition}
Note that if $(B, \cB^{\circ}, \cB^{\filledcirc}, R)$ is a rooted $K_q$-booster, then $B \cup R$ is a $K_q$-booster with decompositions $\cB^{\circ}$ and $\cB^{\filledcirc}\cup\{R\}$.  Similarly, if $B$ is a booster with $K_q$-decompositions $\cB^{\circ}$ and $\cB^{\filledcirc}$, then for any $R \in \cB^{\filledcirc}$, we have that $(B - E(R), \cB^{\circ}, \cB^{\filledcirc}\setminus\{R\}, R)$ is a rooted $K_q$-booster.  Moreover, if $A$ is a $K_q$-omni-absorber for $X$ with decomposition family $\cH$ and decomposition function $\cQ$ and $(B, \cB^{\circ}, \cB^{\filledcirc}, H)$ is a rooted $K_q$-booster for $H \in \cH$ such that $B$ is edge-disjoint from $A$, then $A \cup B$ is also a $K_q$-omni-absorber for $X$ with decomposition family $\cH \cup \cB^{\circ} \cup \cB^{\filledcirc} \setminus\{H\}$ and decomposition function 
\begin{equation*}
    L \mapsto \left\{\begin{array}{l l}
      \cQ(L) \cup \cB^{\circ} \setminus \{H\} & \text{if } H \in \cQ(L)\\
      \cQ(L) \cup \cB^{\filledcirc} & \text{otherwise.}
    \end{array}\right.
\end{equation*}

In the proof of Theorem~\ref{thm:OmniSpread}, we will find a random rooted booster $(B_H, \cB^{\circ}_H, \cB^{\filledcirc}_H, V(H))$ for every $H \in \cH_{A'}$
and apply the above operation for each $H \in \cH_{A'}$ to get a new omni-absorber $A$ satisfying the spread property.
The exponent in the spreadness corresponds to the following parameter of the spread booster.

\begin{definition}\label{def:RootedDensity} Let $q > 1$ be an integer.
\begin{itemize}
    \item For a $K_q$-packing $\cH$ of a graph $G$ and $R \subseteq V(G)$, the \textit{rooted density} of $\cH$ and $R$ is
    \begin{equation*}
        d(\cH, R) \coloneqq \left. |\cH| \middle/ \left|\bigcup_{H\in \cH}V(H)\setminus R\right|,\right.
    \end{equation*}
    and the \textit{maximum rooted density} of $\cH$ and $R$ is
    \begin{equation*}
        m(\cH, R) \coloneqq \max_{\cH' \subseteq \cH}d(\cH', R).
    \end{equation*}
    \item For a rooted $K_q$-booster $(B, \cB^{\circ}, \cB^{\filledcirc}, R)$, the \textit{rooted density} is $
        \max_{\cH \in \{\cB^{\circ}, \cB^{\filledcirc}\}}m(\cH, R)$.
\end{itemize}
\end{definition}

Suppose vertices of a spread booster $(B, B^\circ, B^{\filledcirc}, H)$ were chosen independently with probability $1 / n$.  Then the probability some $K_q$-packing $\cS$ is contained in $\cB^\circ$ or $\cB^{\filledcirc}$ is at most $1 / n^{|\bigcup_{S\in \cS}V(S)\setminus V(H)|} = \left(1 / n^{1/d(\cS, V(H))}\right)^{|\cS|}$, and $d(\cS, V(H))$ is at most the rooted density of the booster if this probability is nonzero.  Hence, it is key to construct spread boosters with small rooted density, as follows.

\begin{lem}\label{lem:sparse-boosters-exist}
  For every integer $q > 2$, there exists a rooted $K_q$-booster with rooted density at most $2/(q-2)$.
\end{lem}

We defer the proof of Lemma~\ref{lem:sparse-boosters-exist} to Section~\ref{sec:spread-boosters}.

\subsection{Vertex-Spread Clique Embedding}

A \textit{partial $K_k$ rooted at $R$} is a graph $T$ with $R \subseteq V(T)$, $|V(T)| = k$, and $E(T) = \binom{V(T)}{2}\setminus \binom{R}{2}$.
We will need to embed spread boosters for each $H \in \cH_{A'}$ so that they are pairwise edge-disjoint and collectively have small maximum degree.  It will be convenient to first embed partial cliques with the same property and choose boosters inside the partial cliques.  Using conditional Lov\'asz Local Lemma, we can find such an embedding in which vertices are chosen in partial cliques in a spread way, as follows.
In this lemma, $\Delta_1(\cH)$ is the maximum $1$-degree of the hypergraph $\cH$; that is, the maximum number of edges containing a given vertex.

\begin{lem}\label{lemma:spread-clique-embedding}
    For all integers $q \geq 3$ and $b \geq 1$, the following holds for sufficiently small $\eps> 0$ and sufficiently large $C > 0$.
    Let $G$ be an $n$-vertex graph with $\delta(G) \geq (1-\varepsilon)n$.
    If $\cH$ is a $q$-uniform hypergraph with vertex set $V(G)$ such that $\Delta_1(\cH) \leq n / C$, then there exists a probability distribution on pairwise edge-disjoint subgraphs $(T_e : e \in E(\cH))$ of $G$ such that $T_e$ is a partial $K_{q+b}$ rooted at $V(e)$ satisfying $\Delta(T) \leq C\Delta_1(\cH)$, where $T \coloneqq \bigcup_{e \in E(\cH)}T_e$, such that
    \begin{equation*}
        \Prob{S_e \subseteq V(T_e)\setminus V(e)~\forall e \in E(\cH)} \leq \left(\frac{(3b)^b}{n}\right)^{\sum_{e\in E(\cH)}|S_e|} \text{ for every } (S_e \subseteq V(G) : e \in E(\cH)). 
    \end{equation*}
\end{lem}
\begin{proof}
    Let $\Delta \coloneqq \Delta_1(\cH)$, let $D \coloneqq \lfloor C\Delta / (2(q + b))\rfloor$, and for each $e \in \cH$, let $\cX_e$ be the set of partial $K_{q + b}$s rooted at $V(e)$.
    Since $\delta(G) \geq (1 - \eps)n$, every set of at most $q + b$ vertices of $G$ have at least $(1 - \eps(q + b))n \geq n / 2$ common neighbors in $G$.  Hence, 
    \begin{enumerate}[(a)]
        \item\label{num-cliques-lower-bound} $|\cX_e| \geq (n/2)^{b}/b! \geq (n/ (2b))^b$ for every $e \in E(\cH)$, and 
        \item\label{num-cliques-upper-bound} $|\{T \in \cX_e : V(T) \supseteq S\} \leq n^{b-|S|}$ for every $e \in E(\cH)$ and $S \subseteq V(G)\setminus V(e)$.
    \end{enumerate}    
  
    First, for each $e \in E(\cH)$, consider choosing $T_e\in \cX_e$ independently and uniformly at random, and choose $i_e \in [D]$ independently and uniformly at random.
    We will use the Lov\'asz Local Lemma to construct the desired probability distribution; to that end, we define the following ``bad events''.  
    
    Forbidding the first type of bad event will ensure $\Delta(T) \leq C\Delta_1(\cH)$.
    Let $J_1$ be the set of $(T_1, T_2, v, i)$, where $T_1 \in \cX_{e_1}$ and $T_{2} \in \cX_{e_2}$ for distinct $e_1, e_2 \in E(\cH)$, and $v \in (V(T_1) \setminus V(e_1)) \cap (V(T_{2}) \setminus V(e_2))$ and $i \in [D]$.  For every $(T_1, T_{2}, v, i) \in J_1$, let $\cE_{(T_1, T_{2}, v, i)}$ be the event that $T_{e_1} = T_{1}$ and $T_{e_2} = T_2$ and $i_{e_1} = i_{e_2} = i$.  

    Forbidding the second type of bad event will ensure the $T_e$'s are edge disjoint.
    Let $J_2$ be the set of $(T_1, T_2, f, i_1, i_2)$, where $T_1 \in \cX_{e_1}$ and $T_2 \in \cX_{e_2}$ for distinct $e_1, e_2 \in E(\cH)$, $f \in E(T_1) \cap E(T_2)$, and $i_1,i_2\in [D]$.  For every $(T_1, T_2, f, i_1,i_2) \in J_2$, let $\cE_{(T_1,T_2,f,i_1,i_2)}$ be the event that $T_{e_j} = T_j$ and $i_{e_j} = i_j$ for $j \in \{1,2\}$.
    Since $|\cX_e| \geq (n/(2b))^{b}$ for every $e \in E(\cH)$,
    \begin{equation*}
        \Prob{\cE_j} \leq (2b / n)^{2b}D^{-2} \text{ for every } j \in J_1 \cup J_2.
    \end{equation*}
    
    Next we construct the dependency graph for the Lov\'asz Local Lemma.  Let $\Gamma$ be the graph with vertex set $J_1 \cup J_2$ where $(T_1, T_2, f,i_1,i_2) \in J_2$ is adjacent to $(T'_1, T'_2, f,i'_1,i'_2) \in J_2$ or $(T'_1, T'_2, v, i) \in J_1$ if either $T_1, T'_1 \in \cX_e$ or $T_2, T'_2 \in \cX_e$ for some $e \in E(\cH)$, and $(T_1, T_2, v, i) \in J_1$ is adjacent to $(T'_1, T'_2, v, i') \in J_1$ if either $T_1, T'_1 \in \cX_e$ or $T_2, T'_2 \in \cX_e$ for some $e \in E(\cH)$.  
    By \ref{num-cliques-upper-bound}, every $j \in J_1 \cup J_2$ has at most 
    \begin{equation*}
        2n^b\binom{q + b}{2}\left(|E(\cH)| \cdot n^{b - 2} + 2\Delta n^{b - 1}\right)D^2 \leq 6D^2 (q + b)^2 n^{2b} / C
    \end{equation*}
    \COMMENT{We have $2$ choices for whether we'll have $T_1, T'_1 \in \cX_{e_1}$ or $T_2, T'_2 \in \cX_{e_2}$, assume wlog the former, then $n^b$ choices for $T'_1 \in \cX_e$, then $\binom{q + b}{2}$ choices for which edge of $T'_1$ plays the role of $f$, and then two types of choices for $T'_2 \in \cX_{e'_2}$ depending on whether $f$ is rooted at $V(e'_2)$ or not.  If it is not rooted at $V(e'_2)$, then we have $|E(\cH)|$ choices for $e'_2$ and at most $n^{b-2}$ choices for $T'_2 \in \cX_{e'_2}$ containing $f$ by \ref{num-cliques-upper-bound}.  If it is rooted at $V(e'_2)$, then there are two choices for the end of $f$, $\Delta_1(\cH)$ choices for $e'_2 \in \cH$ containing that vertex, and $n^{b-1}$ choices for $T'_2 \in \cX_{e'_2}$ containing the other end of $f$ by \ref{num-cliques-upper-bound}.}
    neighbors in $J_2$, where in the second inequality we used $\Delta \leq n / C$ and $|E(\cH)| \leq n\Delta$.  Similarly, by \ref{num-cliques-upper-bound}, every $j \in J_1 \cup J_2$ has at most 
    \begin{equation*}
        2n^b(q + b)\left(|E(\cH)|\cdot n^{b - 1} + \Delta n^b\right)D \leq 12D^2 (q + b)^2n^{2b} / C
    \end{equation*}
    neighbors in $J_1$, where in the second inequality we used $D \geq C\Delta / (3(q + b))$.  Hence, $\Delta(\Gamma) \leq 18D^2(q + b)^2n^{2b} / C$.

    Finally, we claim that the conditional distribution on $\bigcap_{j \in J_1 \cup J_2}\overline{\cE_j}$ is the desired distribution.  Indeed, given $\bigcap_{j \in J_2}\overline{\cE_j}$, we have that $T_{e}$ and $T_{e'}$ are edge-disjoint for distinct $e, e' \in E(\cH)$, and given $\bigcap_{j \in J_1}\overline{\cE_j}$, we have that $\Delta(T) \leq (D + \Delta)(q + b) \leq C\Delta$, as desired.  
    Finally, given $(S_e \subseteq V(G) : e \in E(\cH))$, let $\cE$ be the event that $S_e \subseteq V(T_e) \setminus V(e)$ for all $e \in E(\cH)$. We may assume $S_e \cap V(e) = \emptyset$ and $|S_e| \leq b$ for all $e \in E(\cH)$, or else $\Prob{\cE} = 0$, as desired.  Let $X \coloneqq \{e \in E(\cH) : S_e \neq \emptyset\}$.  By \ref{num-cliques-lower-bound} and \ref{num-cliques-upper-bound}, 
    \begin{equation*}
        \Prob{\cE} \leq \prod_{e\in X}\frac{n^{b - |S_e|}}{(n/(2b))^b} \leq \left(\frac{(2b)^b}{n}\right)^{\sum_{e\in X}|S_e|}.
    \end{equation*}
    Let $N_1$ be the number of $(T_1, T_2, v, i) \in J_1$ such that for some $e \in X$, either $T_1 \in \cX_e$ or $T_2 \in \cX_e$, let $N_2$ be the number of $(T_1, T_2, f, i_1, i_2) \in J_2$ such that for some $e \in X$, either $T_1 \in \cX_e$ or $T_2 \in \cX_e$, and let $N \coloneqq N_1 + N_2$.  Using the same argument used to bound $\Delta(\Gamma)$, we have
    \begin{equation*}
        N_1 \leq |X|6D^2 (q + b)^2 n^{2b}/C\qquad\text{and}\qquad N_2 \leq |X|12D^2 (q + b)^2n^{2b}/C,
    \end{equation*} 
    so $N \leq 18|X|D^2(q + b)^2n^{2b}/C$.  Applying Lemma~\ref{lem:conditionalLLL} with $(2b / n)^{2b}D^{-2}$ playing the role of $p$ implies that 
    \begin{equation*}
        \ProbCond{\cE}{\bigcap_{j \in J_1 \cup J_2}\overline{\cE_j}} \leq \left(\frac{(2b)^b}{n}\right)^{\sum_{e\in X}|S_e|}\exp\left(\frac{108|X|(q + b)^2(2b)^{2b}}{C}\right) \leq \left(\frac{(3b)^b}{n}\right)^{\sum_{e\in E(\cH)}|S_e|},
    \end{equation*}
    as desired.
\end{proof}

\subsection{Proof of the Spread Omni-Absorber Theorem}

Now we can combine Lemmas~\ref{lem:sparse-boosters-exist} and \ref{lemma:spread-clique-embedding} to prove Theorem~\ref{thm:OmniSpread}.

\begin{proof}[Proof of Theorem~\ref{thm:OmniSpread}]
    Let $C'$ be as in Theorem~\ref{thm:omni-absorber}, and let $X$ be a spanning subgraph of $G$ with $\Delta(X) \leq n / C$.
    By Theorem~\ref{thm:omni-absorber}, there exists a $C'$-refined omni-absorber $A' \subseteq G$ for $X$ with decomposition family $\cH_{A'}$ and decomposition function $\cQ_{A'}$ such that $\Delta(A') \leq C'\Delta$, where $\Delta \coloneqq \max\{\Delta(X), \sqrt{n}\log n\}$.  
    Note that $\Delta \leq n / C$ since $C \leq \sqrt{n}/\log n$.
    By Lemma~\ref{lem:sparse-boosters-exist}, there exists a rooted $K_q$-booster $(\tilde{B}, \cB^{\circ}, \cB^{\filledcirc}, R)$ with rooted density at most $2/(q-2)$.  Let $b \coloneqq |V(\tilde B)\setminus V(R)|$.

    Since $A'$ is $C'$-refined and $\Delta(A') \leq C'\Delta \leq C' n / C$, we have $\delta(G - (X\cup A')) \geq (1 - (C'+1)/C)n$ and $|\{H \in \cH_{A'} : V(H) \ni v\}| \leq C'd_{X\cup A'}(v) \leq C'(C' + 1)\Delta \leq C'(C' + 1)n / C$ for every $v \in V(G)$, so $\Delta_1(\cH_{A'}) \leq C'(C' + 1)\Delta \leq C'(C' + 1)n/C$.
    Hence, by Lemma~\ref{lemma:spread-clique-embedding} with $G - (X \cup A')$, $(C' + 1) / C$, $\cH_{A'}$, $C / (2C'(C' + 1))$ playing the roles of $G$, $\eps$, $\cH$, and $C$, respectively, for sufficiently large $C$, there exists a probability distribution on pairwise edge-disjoint subgraphs $(T_H : H \in \cH_{A'})$ of $G$ such that $T_H$ is a partial $K_{q + b}$ rooted at $V(H)$ satisfying $\Delta(T) \leq C\Delta/2$, where $T \coloneqq \bigcup_{H\in\cH_{A'}}T_H$ such that
    \begin{equation}\label{eqn:vertex-spreadness-boosters}
        \Prob{S_H \subseteq V(T_H)\setminus V(H)~\forall H \in \cH_{A'}} \leq \left(\frac{C^{1/(q-2)}}{n}\right)^{\sum_{H\in \cH_{A'}}|S_H|} \text{ for every } (S_H \subseteq V(G) : H \in \cH_{A'}). 
    \end{equation}
    
    For each $H \in \cH_{A'}$, choose a rooted booster $(B_H, \cB_H^{\circ}, \cB_H^{\filledcirc}, H)$ such that $B_H \subseteq T_H$ and $B_H \cong \tilde{B}$.  
    Let $A \coloneqq A' \cup \bigcup_{H\in \cH_{A'}}B_H$, let $\cH_A \coloneqq \bigcup_{H \in \cH_{A'}}(\cB_H^{\circ} \cup \cB_H^{\filledcirc})$, and let $\cQ_A(L) = \bigcup_{H\in\cQ_{A'}(L)}\cB^{\circ}_H \cup \bigcup_{H\in\cH_{A'}\setminus\cQ_{A'}(L)}\cB^{\filledcirc}_H$.  Note that $A$ is a $K_q$-omni-absorber with decomposition family $\cH_A$ and decomposition function $\cQ_A$.
    
    We claim that this probability distribution on $A$, $\cH_A$, and $\cQ_A$ is the desired distribution.  Since $\Delta(T) \leq C\Delta/2$ and $\Delta(A') \leq C'\Delta$, we have $\Delta(A) \leq \Delta(T) + \Delta(A') + \Delta(X) \leq C\Delta$, as required.  Let $\cS$ be a $K_q$-packing of $G$, and let $\cP$ be the set of $((\cS_H, *_H) : H \in \cH_{A'})$ where $(\cS_H : H \in \cH_{A'})$ is a partition of $\cS$ and $*_H \in \{\circ, \filledcirc\}$ for every $H \in \cH_{A'}$.  Note that $|\cP|$ is at most the number of injections from $\cS$ to $\cH_{A'}\times\{\circ,\filledcirc\}$; since $A'$ is $C'$-refined and $\Delta(A') \leq C'\Delta$, we have $2|\cH_{A'}| \leq 2C'(C' + 1)\Delta n \leq n^2 / C^{3/4}$, so
    \begin{equation*}
      |\cP| \leq \left(\frac{n^2}{C^{3/4}}\right)^{|\cS|}.
    \end{equation*}
    
    If there exists $L \subseteq X$ such that $\cS \subseteq \cQ_A(L)$, then there exists a partition $(\cS_H : H \in \cH_{A'})$ of $\cS$ such that $\cS_H \subseteq \cB^{\circ}_H$ for every $H \in \cQ_{A'}(L)$ and $\cS_H \subseteq \cB^{\filledcirc}_H$ for every $H \in \cH_{A'}\setminus\cQ_{A'}(L)$.  In particular, letting $*_H \coloneqq \circ$ for $H \in \cQ_{A'}(L)$ and $*_H\coloneqq\filledcirc$ for $H \in \cH_{A'}\setminus\cQ_{A'}(L)$, we have $\cS_H \subseteq \cB^{*_H}_H$ for every $H \in \cH_{A'}$.  Therefore, by a union bound,
    \begin{equation*}
     \Prob{\bigcup_{L\subseteq X}\{\cS \subseteq \cQ_A(L)\}}  \leq \sum_{((\cS_H, *_H) : H \in \cH_{A'})\in\cP}\Prob{\cS_H \subseteq \cB^{*_H}_H~\forall H \in \cH_{A'}}.
    \end{equation*}

     For every $*\in\{\circ,\filledcirc\}$, and $\cS' \subseteq \cB^{*}$, since $(\tilde B, \cB^{\circ}, \cB^{\filledcirc}, R)$ has rooted density at most $2 / (q - 2)$, we have $|\bigcup_{S\in \cS'}V(S)\setminus V(H)| \geq (q - 2)|\cS'| / 2.$  Therefore, for every $((\cS_H, *_H) : H \in \cH_{A'})\in\cP$, by \eqref{eqn:vertex-spreadness-boosters}, we have 
    \begin{equation*}
      \Prob{\cS_H \subseteq \cB^{*_H}_H~\forall H \in \cH_{A'}} \leq \left(\frac{\sqrt C}{n^{{(q - 2})/{2}}}\right)^{|\cS|}.
    \end{equation*}
    Combining the inequalities above, we have    
    \begin{equation*}
      \Prob{\bigcup_{L\subseteq X}\{\cS \subseteq \cQ_A(L)\}} \le \left(\frac{1}{C^{1/4}n^{(q - 6)/2}}\right)^{|\cS|},
    \end{equation*}
      as desired.
\end{proof}

\section{Construction of Spread Boosters}\label{sec:spread-boosters}

In this section, we construct our spread boosters; namely, we prove Lemma~\ref{lem:sparse-boosters-exist} which says there exists a rooted $K_q$-booster with rooted density at most $2/(q-2)$. First, we prove the following lemma which constructs such a booster except for one special clique which is not counted in the density calculation. 

\begin{lem}\label{lem:spread-booster-w-special-clique}
    For every integer $q > 2$, there exists a $K_q$-booster $B$ with disjoint decompositions $\cB_1$ and $\cB_2$, $S_i \in \cB_i$ for $i \in \{1,2\}$ with $|V(S_1) \cap V(S_2)| = q - 1$ such that
    \begin{itemize}
        \item $m(\cB_1\setminus\{S_1\}, V(S_1)) \leq 2 / (q-2)$ and
        \item $m(\cB_2\setminus\{S_2\}, V(S_1) \cup V(S_2)) \leq 2/(q - 2)$.
    \end{itemize}
\end{lem}
\begin{proof}
    We define the graph $B$ with $V(B) \coloneqq \{v_1, v_2\} \cup [q-1]\times[q-1]$, where $v_1$ and $v_2$ are adjacent to every vertex in $[q - 1] \times [q - 1]$ and $(i,j) \in [q - 1]\times[q - 1]$ is adjacent to $(i',j') \in [q - 1]\times[q - 1]$ if either $i = i'$ (and $j \neq j'$) or $j = j'$ (and $i \neq i'$).
    Note that $B$ is the Cartesian product of $K_{q - 1}$ with $K_{q - 1}$ (or equivalently, the line graph of $K_{q-1,q-1})$ with two dominating vertices added.  The Cartesian product of $K_{q-1}$ with $K_{q - 1}$ has a $K_{q - 1}$-decomposition in which each vertex is contained in exactly two cliques (one ``horizontal'' and one ``vertical'').  We will extend this decomposition to a $K_q$-decomposition of $B$ in two different ways, as follows.  
    For each $i \in [q - 1]$ and $x \in \{1,2\}$, let $Q^{\hoclique}_{i,x} \coloneqq \{v_x\} \cup \{(i, j) : j \in [q - 1]\}]$, and for each $j \in [q - 1]$ and $x \in \{1,2\}$, let $Q^{\vertclique}_{j,x} \coloneqq \{v_x\} \cup \{(i,j) : i \in [q-1]\}$.  Let $\cB_1 \coloneqq \{B[Q^{\hoclique}_{i,1}] : i \in [q - 1]\} \cup \{B[Q^{\vertclique}_{j,2}] : j \in [q - 1]\}$, and let $\cB_2 \coloneqq \{B[Q^{\hoclique}_{i,2}] : i \in [q - 1]\} \cup \{B[Q^{\vertclique}_{j,1}] : j \in [q - 1]\}$.

    Now we have two disjoint $K_q$-decompositions $\cB_1$ and $\cB_2$ of $B$, as desired.  Let  
    $S_1 \coloneqq B[Q^{\hoclique}_{1,1}]$, let $S_2 \coloneqq B[Q^{\hoclique}_{1,2}]$, and note that $S_i \in \cB_i$ for $i \in \{1,2\}$ with $|V(S_1) \cap V(S_2)| = q - 1$, as desired.  It remains to check the rooted densities.  

    To that end, let $\cH \subseteq \cB_1\setminus\{S_1\}$.  Let $\cH^{\hoclique} \coloneqq \cH \cap \{B[Q^{\hoclique}_{i,1}] : i \in [q - 1]\}$, and let $\cH^{\vertclique} \coloneqq \cH\setminus \cH^{\hoclique}$.  Note that
    \begin{equation*}
        \left|\bigcup_{H\in\cH^{\hoclique}}V(H)\setminus (V(S_1) \cup V(S_2))\right| = (q - 1)|\cH^{\hoclique}| \quad\text{and}\quad \left|\bigcup_{H\in\cH^{\vertclique}}V(H)\setminus (V(S_1) \cup V(S_2))\right| = (q - 2)|\cH^{\vertclique}|,
    \end{equation*}
    so
    \begin{equation*}
        d(\cH, V(S_1)) = \frac{|\cH|}{\left|\bigcup_{H\in\cH}V(H)\setminus V(S_1)\right|} \leq \frac{|\cH|}{(q-2)\max\{|\cH^{\hoclique}|, |\cH^{\vertclique}|\}} \leq \frac{2}{q - 2},
    \end{equation*}
    as desired.
    By a nearly identical argument, we can show $m(\cB_2 \setminus \{S_2\}, V(S_1) \cup V(S_2)) \leq 2 / (q - 2)$, so we omit the details.\COMMENT{
    To that end, let $\cH \subseteq \cB_2\setminus\{S_2\}$.  Let $\cH^{\hoclique} \coloneqq \cH \cap \{B[Q^{\hoclique}_{i,2}] : i \in [q - 1]\}$, and let $\cH^{\vertclique} \coloneqq \cH\setminus \cH^{\hoclique}$.  Note that
    \begin{equation*}
        |\bigcup_{H\in\cH^{\hoclique}}V(H)\setminus (V(S_1) \cup V(S_2))| = (q - 1)|\cH^{\hoclique}| \quad\text{and}\quad |\bigcup_{H\in\cH^{\vertclique}}V(H)\setminus (V(S_1) \cup V(S_2))| = (q - 2)|\cH^{\vertclique}|,
    \end{equation*}
    so
    \begin{equation*}
        d(\cH, V(S_1)\cup V(S_2)) = \frac{|\cH|}{|\bigcup_{H\in\cH}V(H)\setminus (V(S_1) \cup V(S_2))|} \leq \frac{|\cH|}{(q-2)\max\{|\cH^{\hoclique}|, |\cH^{\vertclique}|\}} \leq \frac{2}{q - 2},
    \end{equation*}
    as desired.}
\end{proof}

Now, we iteratively use the construction from Lemma~\ref{lem:spread-booster-w-special-clique} to prove Lemma~\ref{lem:sparse-boosters-exist}. Namely, by iteratively layering that construction, we are able to make the special clique more and more disjoint from the rooted clique; once the two cliques are vertex-disjoint, we can then add the special clique back to the rooted density calculation.

\begin{proof}[Proof of Lemma~\ref{lem:sparse-boosters-exist}]
    Let $B$ be a $K_q$-booster with disjoint $K_q$-decompositions $\cB_1$ and $\cB_2$, $S_i \in \cB_i$ for $i \in \{1,2\}$ with
    \begin{enumerate}[label=(\ref*{lem:sparse-boosters-exist}.\arabic*), leftmargin=*]
        \item\label{inductive-booster1} $m(\cB_1\setminus\{S_1\}, V(S_1)) \leq 2 / (q-2)$ and
        \item\label{inductive-booster2} $m(\cB_2\setminus\{S_2\}, V(S_1) \cup V(S_2)) \leq 2/(q - 2)$
    \end{enumerate}
    such that $|V(S_1) \cap V(S_2)|$ is minimized.    
    Such a booster exists by Lemma~\ref{lem:spread-booster-w-special-clique}.  
    
    If $|V(S_1) \cap V(S_2)| = 0$, then we claim that $(B - E(S_1), \cB_2, \cB_1\setminus\{S_1\}, S_1)$ is a rooted $K_q$-booster with rooted density at most $2/(q - 2)$.  Indeed, by \ref{inductive-booster1}, it suffices to show that $m(\cB_2, V(S_1)) \leq 2 / (q - 2)$.  To that end, let $\cH \subseteq \cB_2$.  If $S_2 \notin \cH$, then by \ref{inductive-booster2}, $d(\cH, V(S_1)) \leq d(\cH, V(S_1) \cup V(S_2)) \leq 2/(q - 2)$, as required.  Otherwise, if $S_2 \in \cH$, then by \ref{inductive-booster2},
    \begin{equation*}
         \frac{|\cH| - 1}{|\bigcup_{H\in \cH\setminus\{S_2\}}V(H)\setminus (V(S_1) \cup V(S_2))|} = d(\cH\setminus\{S_2\}, V(S_1) \cup V(S_2)) \leq \frac{2}{q - 2}.
    \end{equation*}
    Moreover,
    \begin{equation*}
        \left|\bigcup_{H\in \cH\setminus\{S_2\}}V(H)\setminus (V(S_1) \cup V(S_2))\right| = \left|\bigcup_{H\in \cH}V(H)\setminus V(S_1)\right| - q,
    \end{equation*}
    and since $|\bigcup_{H\in \cH}V(H)\setminus V(S_1)| \leq q|\cH|$, we have
    \begin{equation*}
        d(\cH, V(S_1)) = \frac{|\cH|}{|\bigcup_{H\in \cH}V(H)\setminus V(S_1)|} \leq \frac{|\cH| - 1}{|\bigcup_{H\in \cH}V(H)\setminus V(S_1)| - q}.
    \end{equation*}
    Combining the inequalities above, we have $d(\cH, V(S_1)) \leq 2 / (q - 2)$, so $m(\cB_2, V(S_1)) \leq 2/(q - 2)$, as claimed.
    
    Therefore, we may assume $V(S_1) \cap V(S_2) \neq \emptyset$.
    By Lemma~\ref{lem:spread-booster-w-special-clique} again, there exists a $K_q$-booster $B'$ with disjoint decompositions $\cB'_1$ and $\cB'_2$, $S'_i \in \cB'_i$ for $i \in \{1,2\}$ with $|V(S'_1) \cap V(S'_2)| = q - 1$ such that 
    \begin{enumerate}[label=(\ref*{lem:sparse-boosters-exist}.\arabic*'), leftmargin=*]
        \item\label{inductive-booster1'} $m(\cB'_1\setminus\{S'_1\}, V(S'_1)) \leq 2 / (q-2)$ and  
        \item\label{inductive-booster2'} $m(\cB'_2\setminus\{S'_2\}, V(S'_1) \cup V(S'_2)) \leq 2/(q - 2)$.
    \end{enumerate}
    By possibly relabelling the vertices of $B'$, we may assume that $S'_1 = S_2$, $V(B') \cap V(B) = V(S_2)$, and $E(B) \cap E(B') = E(S_2)$.  Moreover, since $|V(S'_1) \setminus V(S'_2)| = 1$ and $V(S_1) \cap V(S_2) \neq \emptyset$, we may also assume that $V(S'_1) \setminus V(S'_2) \subseteq V(S_1) \cap V(S_2)$.
    Let $B'' \coloneqq B \cup B'$, $\cB''_1 \coloneqq \cB_1 \cup (\cB'_1\setminus\{S'_1\})$, and $\cB''_2 \coloneqq (\cB_2\setminus \{S_2\}) \cup \cB'_2$,
    and note that $\cB''_1$ and $\cB''_2$ are disjoint $K_q$-decompositions of $B''$.  Moreover, $S_1 \in \cB''_1$ and $S'_2 \in \cB''_2$, and since $V(S'_1) \setminus V(S'_2) \subseteq V(S_1) \cap V(S_2)$, we have $|V(S_1) \cap V(S'_2)| < |V(S_1) \cap V(S_2)|$.  We claim that \ref{inductive-booster1} and \ref{inductive-booster2} hold with $\cB''_1$, $\cB''_2$, $S_1$, and $S'_2$ playing the roles of $\cB_1$, $\cB_2$, $S_1$, and $S_2$, contradicting the choice of $B$ with $|V(S_1) \cap V(S_2)|$ minimum.

    To that end, let $\cH \subseteq \cB''_1\setminus\{S_1\}$. Let $\cH_1 \coloneqq \cH \cap \cB_1$, and let $\cH_2 \coloneqq \cH\setminus \cH_1$.  By \ref{inductive-booster1},
    \begin{equation*}
        \frac{|\cH_1|}{|\bigcup_{H \in \cH_1}V(H)\setminus V(S_1)|} = d(\cH_1, V(S_1)) \leq \frac{2}{q-2},
    \end{equation*}
    and by \ref{inductive-booster1'}
    \begin{equation*}
        \frac{|\cH_2|}{|\bigcup_{H \in \cH_2}V(H)\setminus V(S'_1)|} = d(\cH_2, V(S'_1)) \leq \frac{2}{q-2}.
    \end{equation*}
    Therefore
    \begin{equation*}
        |\cH_1| + |\cH_2| \leq \frac{2}{q - 2}\left(\left|\bigcup_{H \in \cH_1}V(H)\setminus V(S_1)\right| + \left|\bigcup_{H \in \cH_2}V(H)\setminus V(S'_1)\right|\right),
    \end{equation*}
    and since $V(B) \cap V(B') = V(S_2) = V(S'_1)$, we have
    \begin{equation*}
        \left|\bigcup_{H \in \cH}V(H)\setminus V(S_1)\right| = \left|\bigcup_{H \in \cH_1}V(H)\setminus V(S_1)\right| + \left|\bigcup_{H \in \cH_2}V(H)\setminus V(S'_1)\right|.
    \end{equation*}
    Combining the inequalities above, we have
    \begin{equation*}
        d(\cH, V(S_1)) = \frac{|\cH_1| + |\cH_2|}{|\bigcup_{H \in \cH}V(H)\setminus V(S_1)|} \leq \frac{2}{q-2},
    \end{equation*}
    so $m(\cB''_1, V(S_1)) \leq 2 / (q - 2)$, as required.
    By a nearly identical argument, we can show $m(\cB''_2 \setminus \{S'_2\}, V(S_1) \cup V(S'_2)) \leq 2 / (q - 2)$, so we omit the details.\COMMENT{
    To that end, let $\cH \subseteq \cB''_2\setminus\{S'_2\}$. Let $\cH_1 \coloneqq \cH \cap \cB_2$, and let $\cH_2 \coloneqq \cH\setminus \cH_1$.  By \ref{inductive-booster2},
    \begin{equation*}
        \frac{|\cH_1|}{|\bigcup_{H \in \cH_1}V(H)\setminus (V(S_1) \cup V(S_2)|} = d(\cH_1, V(S_1) \cup V(S_2)) \leq \frac{2}{q-2},
    \end{equation*}
    and by \ref{inductive-booster2'}
    \begin{equation*}
        \frac{|\cH_2|}{|\bigcup_{H \in \cH_2}V(H)\setminus (V(S'_1) \cup V(S'_2)|} = d(\cH_2, V(S'_1)) \leq \frac{2}{q-2}.
    \end{equation*}
    Therefore
    \begin{equation*}
        |\cH_1| + |\cH_2| \leq \frac{2}{q - 2}\left(|\bigcup_{H \in \cH_1}V(H)\setminus (V(S_1) \cup V(S_2)| + |\bigcup_{H \in \cH_2}V(H)\setminus (V(S'_1) \cup V(S'_2)|\right),
    \end{equation*}
    and since $V(B) \cap V(B') = V(S_2) = V(S'_1)$, we have
    \begin{equation*}
        |\bigcup_{H \in \cH}V(H)\setminus (V(S_1) \cup V(S'_2)| = |\bigcup_{H \in \cH_1}V(H)\setminus (V(S_1) \cup V(S_2)| + |\bigcup_{H \in \cH_2}V(H)\setminus (V(S'_1) \cup V(S'_2)|.
    \end{equation*}
    Combining the inequalities above, we have
    \begin{equation*}
        d(\cH, V(S_1) \cup V(S'_2)) = \frac{|\cH_1| + |\cH_2|}{|\bigcup_{H \in \cH}V(H)\setminus (V(S_1) \cup V(S'_2)|} \leq \frac{2}{q-2},
    \end{equation*}
    so $m(\cB''_1, V(S_1)\cup V(S'_2)) \leq 2 / (q - 2)$, as required.
    }
\end{proof}

\section{Spread Nibble with Reserves}\label{s:SpreadNibble}

The next key ingredient for our proof of Theorem~\ref{thm:ExistenceSpread} is a spread version of nibble with reserves.  As this is not the bottleneck, we can get away with a simple random sparsification argument.

\begin{thm}\label{thm:spread-nibble}
  For all $r, \gamma, \beta > 0$, the following holds for sufficiently large $D$ and sufficiently small $\alpha > 0$.
  Let $G$ be an $r$-uniform (multi)-hypergraph such that $G$ is the edge-disjoint union of $G_1$ and $G_2$ where $G_2=(A,B)$ is a bipartite hypergraph and $G_1$ is a hypergraph with $V(G_1)\cap V(G_2) = A$.
  If $G$ has codegrees at most $D^{1-\beta}$, every vertex of $B$ has degree at most $D$ in $G_2$, every vertex of $A$ has degree at least $D^{1-\alpha}$ in $G_2$, every vertex of $G_1$ has degree at most $D$ in $G_1$, and every vertex of $A$ has degree at least $D\left(1- D^{-\beta}\right)$ in $G_1$,
  then there exists a probability distribution on $A$-perfect matchings $M$ of $G$ such that
  \begin{equation*}
    \Prob{S \subseteq M} \leq \left(\frac{1}{D^{1 - \gamma}}\right)^{|S|} \text{for every $S \subseteq E(G)$}.
  \end{equation*}
\end{thm}

\begin{proof}
  We may assume without loss of generality that $\gamma < \beta$, or else we replace $\gamma$ with $\beta$.  Moreover, we may assume $\gamma$ is sufficiently small with respect to $\beta$, $\alpha$ is sufficiently small with respect to $\gamma$, and $1 / D$ is sufficiently small with respect to $\alpha$.  First, consider the random spanning hypergraphs $G'_1 \subseteq G_1$, and $G'_2 \subseteq G_2$ obtained by including each $e \in E(G_i)$ in $G'_i$ independently with probability $D^{\gamma/2 - 1}$ for each $i \in \{1,2\}$, and let $G' \subseteq G'_1 \cup G'_2$.
  We define the following events.
  \begin{itemize}
  \item For every $\{u,v\} \in \binom{V(G)}{2}$, let $\cE_{\{u,v\}}$ be the event that $d_{G'}(\{u,v\}) > \log^2 D$.
  \item For every $b \in B$, let $\cE_b$ be the event that $d_{G'_2}(b) > D^{\gamma/2} + D^{\gamma/3}$.
  \item For every $v \in V(G_1)$, let $\cE_{(v, 0)}$ be the event that $d_{G'_1}(v) > D^{\gamma/2} + D^{\gamma/3}$.
  \item For every $a \in A$, let $\cE_{(a,1)}$ be the event that $d_{G'_1}(a) < D^{\gamma/2}(1 - D^{-\gamma / 6})$.
  \item For every $a \in A$, let $\cE_{(a,2)}$ be the event that $d_{G'_2}(a) < D^{\gamma/2 - \alpha} - D^{\gamma/3}$.
  \end{itemize}
  Let $\Gamma$ be the graph with vertex set $\binom{V(G)}{2} \cup B \cup (V(G_1)\times\{0\}) \cup (A\times[2])$ where $\{u,v\} \in \binom{V(G)}{2}$ is adjacent to $\{u',v'\} \in \binom{V(G)}{2}$ if $G$ contains an edge containing $\{u,v, u', v'\}$ and adjacent to $b \in B$ if $G$ contains an edge containing $\{u,v, b\}$.  Other adjacencies in $\Gamma$ are defined similarly so that $S_j \cap S_{j'} = \emptyset$ for distinct non-adjacent $j,j'\in V(\Gamma)$, where $S_j$ and $S_{j'}$ are the sets of trials determining whether an edge is included in $G'_1$ or $G'_2$ upon which $\cE_j$ and $\cE_{j'}$ depend, respectively.  Note that $\Delta(\Gamma) \leq 5r^2D$.
  By a standard application of the Chernoff Bounds (see e.g.~\cite[Theorem A.1.12]{AS16} and \cite[Corollary 2.3]{JLR00}), $\Prob{\cE_j} \leq \exp(\log^2(D))$ for every $j \in V(\Gamma)$\COMMENT{
  \begin{itemize}
      \item $\Expect{d_{G'}(\{u,v\})}  \leq (D^{1 - \beta})D^{\gamma/2 - 1} = D^{\gamma/2 - \beta}$, so \begin{equation*}
          \Prob{d_{G'}(\{u,v\}) > (\log^2D \cdot D^{\beta-\gamma/2})D^{\gamma/2 - \beta}} \leq (\log^2D \cdot D^{\beta-\gamma/2} / e)^{-\log^2 D} \leq \exp(-\log^2 D);
      \end{equation*}
      \item $\Expect{d_{G'_2}(b)} \leq D^{\gamma/2}$, so
      \begin{equation*}
          \Prob{d_{G'_2}(b) > D^{\gamma/2} + D^{\gamma/2}\cdot D^{-\gamma/6}} \leq \exp(-D^{-2\gamma/6}\cdot D^{\gamma/2}/3) \leq \exp(-\log^2 D);
      \end{equation*}
      \item $\Expect{d_{G'_1}(v)} \leq D^{\gamma/2}$, so
      \begin{equation*}
          \Prob{d_{G'_1}(v) > D^{\gamma/2} + D^{\gamma/2}\cdot D^{-\gamma/6}} \leq \exp(-D^{-2\gamma/6}\cdot D^{\gamma/2}/3) \leq \exp(-\log^2 D);
      \end{equation*}
      \item $\Expect{d_{G'_1}(a)} \geq D^{\gamma/2}(1 - D^{-\beta})$, and $D^{\gamma/2}(1 - D^{-\beta})(1 - D^{-\gamma/6}/2) \geq D^{\gamma/2}(1 - D^{-\gamma/6})$ so 
      \begin{equation*}
          \Prob{d_{G'_1}(a) < D^{\gamma/2}(1 - D^{-\beta})(1 - D^{-\gamma/6}/2)} \leq \exp(-D^{-2\gamma/6}\cdot D^{\gamma/2}/12) \leq \exp(-\log^2 D);
      \end{equation*}
      \item $\Expect{d_{G'_2}(a)} \leq D^{\gamma/2 - \alpha}$, so
      \begin{equation*}
          \Prob{d_{G'_2}(a) < D^{\gamma/2 - \alpha} - D^{\gamma/2 - \alpha}\cdot D^{-\gamma/6+\alpha}} \leq \exp(-D^{2(-\gamma/6+\alpha)}\cdot D^{\gamma/2}/3) \leq \exp(-\log^2 D);
      \end{equation*}
  \end{itemize}
  }.

  For every outcome $G'$ for which $\bigcap_{j \in V(\Gamma)}\overline{\cE_j}$ holds, by Theorem~\ref{thm:NibbleReserves} applied with $G'$, $G'_1$, $G'_2$, $D^{\gamma / 2} + D^{\gamma/3}$, $1/4$, and $3\alpha / \gamma$ playing the roles of $G$, $G_1$, $G_2$, $D$, $\beta$, and $\alpha$, respectively\COMMENT{
  \begin{itemize}
  \item $(D^{\gamma/2} + D^{\gamma/3})^{1 - 1/4} \geq \log^2 D$, so the codegrees are small enough;
  \item $(D^{\gamma/2} + D^{\gamma/3})^{1 - 3\alpha/\gamma} \leq D^{\gamma/2 - \alpha} - D^{\gamma/3}$, so degrees of vertices in $A$ are large enough in $\cG'_2$;
  \item Since
  \begin{align*}
      D^{\gamma/2} + D^{\gamma/3} - (D^{\gamma/2} + D^{\gamma/3})^{1-1/4} &\leq D^{\gamma/2} - D^{\gamma/3} &\text{iff}\\
      2D^{\gamma/3} &\leq (D^{\gamma/2} + D^{\gamma/3})^{3/4},
  \end{align*}
  we have $(D^{\gamma/2} + D^{\gamma/3})(1 - (D^{\gamma/2} + D^{\gamma/3})^{-1/4}) \leq D^{\gamma/2}(1 - D^{-\gamma/6})$, so degrees of vertices in $A$ are large enough in $\cG'_1$;  
  \end{itemize}
  }, $G'$ has an $A$-perfect matching $M$.
  We claim that the conditional distribution on $\bigcap_{j \in V(\Gamma)}\overline{\cE_j}$ yields the desired distribution.
  To that end, let $S \subseteq E(G)$, and let $\cE$ be the event that $S \subseteq E(G')$.  Note that there are at most $5r^2 D|S|$ choices of $j \in V(\Gamma)$ for which $\cE_j$ depends on whether $S \subseteq E(G')$.  Hence, by Lemma~\ref{lem:conditionalLLL},      
  \begin{equation*}
    \ProbCond{\cE}{\bigcap_{j \in V(G)}\overline{\cE_j}} \leq \left(D^{\gamma/2 - 1}\right)^{|S|}\exp(30r^2 D|S|\exp(-\log^2 D)) \leq \left(2D^{\gamma/2 - 1}\right)^{|S|}.
  \end{equation*}
  Since $2D^{\gamma/2 - 1} \leq 1 / D^{1 - \gamma}$ and $S\subseteq M$ only if $S \subseteq E(G')$, the result follows.
\end{proof}

\section{Proof of Theorem~\ref{thm:spread-decomposition}}\label{s:MainProof}

Now prove Theorem~\ref{thm:spread-decomposition}, which as mentioned, implies Theorem~\ref{thm:ExistenceSpread}.
\begin{proof}[Proof of Theorem~\ref{thm:spread-decomposition}]
    Throughout the proof we assume $n$ is sufficiently large and $\beta$ sufficiently small for various inequalities to hold.  In addition, we assume $n$ is sufficiently large with respect to $1 / \beta$.  
  Let $G \cong K_n$, let $C \coloneqq n^{8\beta}$, and let $p \coloneqq C^{-2}/2$.  By Lemma~\ref{lem:reservoir}, there exists $X \subseteq G$ such that $\Delta(X) \leq \Delta \coloneqq 2pn = n/C^2$ and for all $e \in E(G) \setminus E(X)$, there exist at least $p^{\binom{q}{2}}n^{q - 2}$ $K_q$'s in $X \cup \{e\}$ containing $e$.
  By Theorem~\ref{thm:OmniSpread}, there exists a probability distribution on $K_q$-omni-absorbers $A$ for $X$ with decomposition family $\cH_A$ and decomposition function $\cQ_A$ such that $\Delta(A) \leq C\Delta$ and 
  \begin{equation}\label{eqn:omni-absorber-spreadness}
      \Prob{\bigcup_{L\subseteq X}\{\cS \subseteq \cQ_A(L)\}}  \le \left(\frac{1}{n^{(q - 6)/2 + 2\beta}}\right)^{|\cS|} \text{ for all $K_q$-packings $\cS$ of $G$.}
  \end{equation}
  
  Let $\cA$ be the set of omni-absorbers supported by this distribution.  For each omni-absorber $A \in \cA$, let $J_A \coloneqq G - (X \cup A)$, and note that $\delta(G) \geq (1 - (C+1)/C^2)n$.  Hence, by Lemma~\ref{lem:RegBoost} applied to $J_A$, there exists a hypergraph $\cD^A_1 \subseteq \mathrm{Design}(J_A, K_q)$ that is $((1/2 \pm n^{-(q-2)/3})\left.\binom{n - 2}{q - 2}\middle.\right.)$-regular.
  Let $\cD^A_2$ be the $\binom{q}{2}$-uniform hypergraph where
  \begin{equation*}
      V(\cD^A_2) \coloneqq E(J_A) \cup X
      \qquad\text{and}\qquad
      E(\cD^A_2)\coloneqq\{S \subseteq V(\cD^A_2) : |S \cap E(J_A)| = 1 \text{ and } G[S] \cong K_q\}.
  \end{equation*}
  We will apply Theorem~\ref{thm:spread-nibble} with $\cD^A_1$ and $\cD^A_2$ playing the roles of $G_1$ and $G_2$, respectively.  To that end, note that $\cD^A_2$ is bipartite with bipartition $(E(J_A), X)$, that $V(\cD^A_1) \cap V(\cD^A_2) = E(J_A)$, and that $\cD^A_1$ and $\cD^A_2$ are edge disjoint, and let $\cD^A \coloneqq \cD^A_1 \cup \cD^A_2$ and $D \coloneqq (1/2 + n^{-(q-2)/3})\binom{n-2}{q-2}$.  Note that $n^{q-2}/(3(q-2)!) \leq D \leq n^{q - 2}$.
  The maximum codegree of $\cD^A$ is at most the maximum codegree of $\mathrm{Design}(K_n, K_q)$ which is at most $\binom{n - 3}{q - 3} \leq D^{1 - 1/(2q-5)}$ since $D \geq n^{q - 5/2}$.
  Every $e \in X$ is in at most $\Delta(X)n^{q-3} \leq n^{q-2}/C^2 \leq D$ hyperedges $S\subseteq E(\cD^A_2)$, so $d_{\cD^A_2}(e) \leq D$ for every $e \in X$.  
  Every $e \in E(J_A)$ is in at least $p^{\binom{q}{2}}n^{q - 2} \geq n^{q - 2 - 8q^2\beta} \geq D^{1 - 8q^2\beta}$ hyperedges $S \subseteq E(\cD^A_2)$ by the choice of $X$, so $d_{\cD^A_2}(e) \geq D^{1 - 8q^2\beta}$ for every $e \in E(J_A)$.
  Every $e \in \cD^A_1$ satisfies $d_{\cD^A_1}(e) \leq D$ and $d_{\cD^A_1}(e) \geq (1/2 - n^{-(q-2)/3})\binom{n-2}{q-2} \geq (1 - 2n^{-(q-2)/3})D \geq (1 - D^{-1/4})D$ by the choice of $D$ and $\cD^A_1$, where in the last inequality we used that $D \geq n^{q-2} / (3(q-2)!)$.
  Hence, we may apply Theorem~\ref{thm:spread-nibble} with $\binom{q}{2}$, $1/q$, $\min\{1/(2q-5), 1/4\}$, $D$, and $8q^2\beta$ playing the roles of $r$, $\gamma$, $\beta$, $D$, and $\alpha$, respectively, to obtain a probability distribution on $E(J_A)$-perfect matchings $M_A$ of $\cD^A$, which correspond to $K_q$-packings $\cM_A$ of $G - A$ covering every edge of $J_A$, such that
  \begin{equation}\label{eqn:nibble-spread}
      \ProbCond{\cS \subseteq \cM_A}{A = A'} \le \left(\frac{1}{n^{q-3}}\right)^{|\cS|} \text{ for all $K_q$-packings $\cS$ of $G$ and $A' \in \cA$}.
  \end{equation}
  For every such matching, we let $L_{A, M_A} \coloneqq X \setminus \bigcup M_A$, and we choose the $K_q$-decomposition $\cM_A \cup \cQ_A(L_{A, M_A})$.  
  
  We claim that this probability distribution on $K_q$-decompositions is $(n^{-(q-6)/2 - \beta})$-spread.  To that end, let $\cS$ be a $K_q$-packing of $G$, and note that
    \begin{align*}
        \Prob{\cS \subseteq \cM_A \cup \cQ_A(L_{A, M_A})} &= \sum_{\cS' \subseteq \cS}\Prob{\cS' \subseteq \cM_A \text{ and } \cS\setminus \cS'\subseteq \cQ_A(L_{A, M_A})}\\
        &\leq\sum_{\cS'\subseteq \cS}\Prob{\cS' \subseteq \cM_A \text{ and } 
      \bigcup_{L\subseteq X}\{\cS\setminus \cS' \subseteq \cQ_A(L)\}}\\
        &=\sum_{\cS' \subseteq \cS}\ProbCond{\cS' \subseteq \cM_A}{\bigcup_{L\subseteq X}\{\cS\setminus \cS' \subseteq \cQ_A(L)\}}\Prob{\bigcup_{L\subseteq X}\{\cS\setminus \cS' \subseteq \cQ_A(L)\}}.
    \end{align*}
    By the law of total probability and \eqref{eqn:nibble-spread}, for every $\cS' \subseteq \cS$,
    \begin{align*}
        \ProbCond{\cS' \subseteq \cM_A}{\bigcup_{L\subseteq X}\{\cS\setminus \cS' \subseteq \cQ_A(L)\}} &= \sum_{A'\in \cA}\ProbCond{\cS' \subseteq \cM_A}{A = A'}\ProbCond{A = A'}{\bigcup_{L\subseteq X}\{\cS\setminus \cS' \subseteq \cQ_A(L)\}}\\
        &\leq \left(\frac{1}{n^{q-3}}\right)^{|\cS'|} \sum_{A' \in \cA}\ProbCond{A = A'}{\bigcup_{L\subseteq X}\{\cS\setminus \cS' \subseteq \cQ_A(L)\}}\\
        &= \left(\frac{1}{n^{q-3}}\right)^{|\cS'|}.
    \end{align*}
    For every $\cS' \subseteq \cS$, by \eqref{eqn:omni-absorber-spreadness}, we have
    \begin{equation*}
        \Prob{\bigcup_{L\subseteq X}\{\cS\setminus \cS' \subseteq \cQ_A(L)\}} \leq \left(\frac{1}{n^{(q - 6)/2 + 2\beta}}\right)^{|\cS|-|\cS'|}.
    \end{equation*}
    Combining the inequalities above, since $q - 3 \geq (q - 6)/2 + 2\beta$ and $2/n^{(q-6)/2 + 2\beta} \leq 1/n^{(q-6)/2 + \beta}$, we have
    \begin{equation*}
        \Prob{\cS \subseteq \cM_A \cup \cQ_A(L_{A, M_A})} 
        \leq \sum_{\cS' \subseteq \cS}\left(\frac{1}{n^{q-3}}\right)^{|\cS'|}\left(\frac{1}{n^{(q - 6)/2 + 2\beta}}\right)^{|\cS|-|\cS'|} 
        \leq \left(\frac{1}{n^{(q - 6)/2 + \beta}}\right)^{|\cS|},
    \end{equation*}
    as desired.
\end{proof}


\end{document}